\UseRawInputEncoding
\documentclass[oneside, 11pt]{amsart}
\usepackage{amsmath}
\usepackage{amssymb}
\usepackage{dsfont}
\usepackage{color}
\usepackage[usenames,dvipsnames,x11names,svgnames]{xcolor}
\usepackage[colorlinks=true,linkcolor=black,citecolor=black, urlcolor=Blue]{hyperref}
\usepackage{amsthm}
\usepackage{amscd}
\usepackage{geometry}
\usepackage{mathrsfs}
\usepackage{mathtools}

\newtheorem{theorem}{Theorem}

\newtheorem{corollary}{Corollary}

\theoremstyle{definition}

\newtheorem*{remark}{Remark}

\newcommand{\be}{\begin{equation*}}
\newcommand{\ee}{\end{equation*} }

\newcommand{\ben}{\begin{equation}}
\newcommand{\een}{\end{equation} }

\newcommand{\bs}{\begin{split}}
\newcommand{\es}{\end{split}}

\newcommand{\bmu}{\begin{multline*}}
\newcommand{\emu}{\end{multline*}}

\newcommand{\bmun}{\begin{multline}}
\newcommand{\emun}{\end{multline}}

\begin{document}



\title[]{On the de Bruijn--Newman constant: a new approach}

\author{Xiao-Jun Yang$^{1,2}$}

\email{dyangxiaojun@163.com; xjyang@cumt.edu.cn}

\address{$^{1}$ School of Mathematics,and State Key Laboratory for Geo-Mechanics and
Deep Underground Engineering, China University of Mining and Technology, Xuzhou 221116, China}
\address{$^{2}$ Department of Mathematics, Faculty of Science, King Abdulaziz University P.O. Box 80257, Jeddah 21589, Saudi Arabia}

\begin{abstract}
The conjecture of Newman, proposed in 1976 by Newman, states that all zeros of
$\Xi_\aleph \left( \lambda \right)$ are real for $\aleph \in \mathbb{R}$.
Its equivalent statement is that
$\mathbb{M}_\aleph \left( \tau \right)$ has purely imaginary zeros
for $\aleph \in \mathbb{R}$.
It is well known that
$\mathbb{M}_\aleph \left( \tau \right)$ is an even entire function of order one.
This article addresses the product representation for $\mathbb{M}_\aleph \left( \tau \right)$ by the works of
Hadamard and Csordas, Norfolk and Varga.
We establish a new class of $\mathbb{M}_\aleph \left( \tau \right)$ by its series and product.
Based on the obtained result, we prove
that it has only purely imaginary zeros for $\aleph \in \mathbb{R}$. This
implies that the conjecture of Newman is true.
\end{abstract}

\maketitle
%\tableofcontents
%%%%%%%%%%%%%%%%%%%%%%%%%%%%%%%%%%%%%%%%%%%%%%%%%%%%%%%%%%%%%%%%%%%%%%%%%%%%%%%%%%%%%%%%%%%
%\\maketitle                                %！！Show title
%\tableofcontents                             %！！Content(1)
     %\\section{}                             %！！1 title
      %\subsection{}                         %！！2 title
       %\subsubsection{}                     %！！3 title
%%%%%%%%%%%%%%%%%%%%%%%%%%%%%%%%%%%%%%%%%%%%%%%%%%%%%%%%%%%%%%%%%%%%%%%%%%%%%%%%%%%%%%%%%%%%%%%%%%%%%%

%%%%%%%%%%%%%%%%%%%%%%%%%%%%%%%%%%%%%%%%%%%%%%%%%%%%%%%%%%%%%%%%%%%%%%%%%%%%%%%%%%%%%%%%%%%%%%%%%%%%%%
%%%%%%%%%%%%%%%%%%%%%%%%%%%%%%%%%%%%%%%%%%%%%%%%%%%%%%%%%%%%%%%%%%%%%%%%%%%%%%%%%%%%%%%%%%%%%%%%%%%%%%
\section{Introduction} {\label{sec:1}}

In 1950 de Bruijn \cite{1} introduced a family of the functions $\Xi
_\aleph \left( \lambda \right):\mathbb{C}\to \mathbb{C}$ for $\aleph \in \mathbb{R}$
by the Fourier cosine integral
\begin{equation}
\label{eq1}
\Xi _\aleph \left( \lambda \right)=\int\limits_0^\infty {\exp \left(
{-\aleph x^2} \right)G\left( x \right)\cos \left( {\lambda x} \right)dx} ,
\end{equation}
where
\begin{equation}
\label{eq2}
G\left( x \right)=\sum\limits_{m=1}^\infty {\left[ {2\pi ^2m^4\exp \left(
{9x} \right)-3\pi m^2\exp \left( {5x} \right)} \right]} \exp \left( {-\pi
m^2\exp \left( {4x} \right)} \right).
\end{equation}
Here, we denote the sets of the real and complex numbers by $\mathbb{R}$ and $\mathbb{C}$,
respectively.

It has been reported by Rodgers and Tao that the even entire function (\ref{eq1})
has the functional equation \cite{2}
\begin{equation}
\label{eq3}
\Xi _\aleph \left( {\overline \lambda } \right)=\overline {\Xi _\aleph
\left( \lambda \right)}
\end{equation}
and satisfies the backward heat equation \cite{3}
\begin{equation}
\label{eq4}
\partial _\aleph \Xi _\aleph \left( \lambda \right)+\partial _{\lambda
\lambda } \Xi _\aleph \left( \lambda \right)=0.
\end{equation}
In 1976 Newman \cite{4} considered that there exists an absolute constant
\begin{equation}
\label{eq5}
\infty <\aleph \le \frac{1}{2},
\end{equation}
now known as the de Bruijn--Newman constant $\aleph $, with the case that
$\Xi _\aleph \left( \lambda \right)$ has purely real zeros.

Newman \cite{4} conjectured $\aleph \ge 0$ with the case that $\Xi _\aleph \left(
\lambda \right)$ has purely real zeros. The recent progress in the lower
bound on $\aleph $ is presented in Table \ref{T1} \cite{2}. The recent result on the
upper bound on $\aleph $ is showed in Table \ref{T2}. Using the result of P\'{o}lya
\cite{4a}, we find that $\aleph =0$ implies that $\Xi _0 \left( \lambda
\right)$ has purely real zeros. The conjecture of Newman \cite{4} states that all
zeros of $\Xi _\aleph \left( \lambda \right)$ are real for $\aleph \in \mathbb{
R}$.

\begin{table*}[]
\centering
\caption{The results in the lower bound on $\aleph $ \cite{2}}
\begin{tabular}{l*{2}c}
     \hline
     Lower bound on $\aleph $  & References \\
     \hline
     $-\infty $ &
Newman \cite{4} \\

$-50$ &
Csordas--Norfolk--Varga \cite{5} \\

$-5$ &
te Riele \cite{6} \\

$-0.385$ &
Norfolk--Ruttan--Varga \cite{7} \\

$-0.0991$ &
Csordas--Ruttan--Varga \cite{8} \\

$-0.379\times 10^{-6}$ &
Csordas--Smith--Varga \cite{9} \\

$-5.895\times 10^{-9}$ &
Csordas--Odlyzko--Smith--Varga \cite{10} \\

$-2.63\times 10^{-9}$ &
Odlyzko \cite{11} \\

$-1.15\times 10^{-11}$ &
Saouter--Gourdon--Demichel \cite{12} \\
         \hline
\end{tabular} \label{T1}
\end{table*}

\begin{table*}[]
\centering
\caption{The results in the upper bound on $\aleph $.}
\begin{tabular}{l*{2}c}
     \hline
     Upper bound on $\aleph $ & References \\
     \hline
     $0.5$ &
Ki, Kim and Lee \cite{13} \\
\hline
$0.22$ &
Polymath \cite{14} \\
\hline
$+\infty $ &
Rodgers and Tao \cite{2} \\
\hline
\end{tabular} \label{T2}
\end{table*}

In this article we consider the following case of (\ref{eq1}):
\begin{equation}
\label{eq6}
-\infty <\aleph <\infty ,
\end{equation}
which can be equivalently transferred into a family of the
function $\mathbb{M}_\aleph \left( \tau \right):\mathbb{
C}\to \mathbb{C}$ for $\aleph \in \mathbb{R}$ by the integral \cite{5}
\begin{equation}
\label{eq7}
\mathbb{M}_\aleph \left( \tau \right)=\int\limits_0^\infty {\exp \left(
{-\aleph x^2} \right)G\left( x \right)\cosh \left( {\tau x} \right)dx}
\end{equation}
with the property that $\mathbb{M}_\aleph \left( \tau \right)$ has purely
imaginary zeros.

Obviously, substituting
\begin{equation}
\label{eq8}
\tau =i\lambda
\end{equation}
into (\ref{eq7}), this implies that
\begin{equation}
\label{eq9}
\Xi _\aleph \left( \lambda \right)=\mathbb{M}_\aleph \left( i\lambda \right).
\end{equation}
It was proved by Csordas, Norfolk and Varga \cite{5} that the
function $\mathbb{M}_\aleph \left( \tau \right)$ is an
even entire function of order $\beta =1$.

With the help of (\ref{eq6}), (\ref{eq7}), and (\ref{eq8}), we need to prove the equivalent
representation for the conjecture of Newman as follows:

\begin{theorem}
\label{TH1}
The function $\mathbb{M}_\aleph \left( \lambda \right)$
has only purely imaginary zeros for $\aleph \in \mathbb{
R}$.
\end{theorem}

Motivated by the idea, we plan to set up a class of (\ref{eq7}),
expressed by the product and series, to prove
Theorem \ref{TH1}. The outline of the paper is given as follows. In Section
\ref{sec:2} we consider the product and series presentations and order of (\ref{eq7}). In
Section \ref{sec:3} we give the six steps to prove Theorem \ref{TH1}. In Section \ref{sec:4} we
finally draw our conclusion.

\section{The series and product for (\ref{eq7})}{\label{sec:2}}

\begin{theorem}
\label{TH2}
There exists
\begin{equation}
\label{eq10}
\mathbb{M}_\aleph \left( \tau \right)=\sum\limits_{\gamma =0}^\infty {\alpha
_{2\gamma } \tau ^{2\gamma }} ,
\end{equation}
where
\begin{equation}
\label{eq11}
\alpha _{2\gamma } =\frac{1}{\left( {2\gamma } \right)!}\int\limits_0^\infty
{\exp \left( {-\aleph x^2} \right)G\left( x \right)x^{2\gamma }dx} .
\end{equation}
\end{theorem}

\begin{proof}
Using the fact
\begin{equation}
\label{eq12}
\cosh \left( x \right)=\sum\limits_{\gamma =0}^\infty {\frac{x^{2\gamma
}}{\left( {2\gamma } \right)!}} ,
\end{equation}
(\ref{eq7}) can be expressed as
\begin{equation}
\label{eq13}
\begin{aligned}
\mathbb{M}_\aleph \left( \tau \right)&=\int\limits_0^\infty {\exp \left(
{-\aleph x^2} \right)G\left( x \right)\cosh \left( {\tau x} \right)dx} \\
 &=\int\limits_0^\infty {\exp \left( {-\aleph x^2} \right)G\left( x
\right)\left[ {\sum\limits_{\gamma =0}^\infty {\frac{\left( {\tau x}
\right)^{2\gamma }}{\left( {2\gamma } \right)!}} } \right]dx} \\
 &=\sum\limits_{\gamma =0}^\infty {\alpha _{2\gamma } \tau ^{2\gamma }} , \\
\end{aligned}
\end{equation}
where $\alpha _{2\gamma } $ is defined by (\ref{eq11}).

We hence complete the proof.
\end{proof}

\begin{theorem}
\label{TH3}
There exists
\begin{equation}
\label{eq14}
\mathbb{M}_\aleph \left( \tau \right)=\mathbb{M}_\aleph \left( 0
\right)\prod\limits_{\Im \left( {\sigma _\ell } \right)>0} {\left(
{1-\frac{\tau ^2}{\sigma _\ell ^2 }} \right)} ,
\end{equation}
where this product runs over all zeros $\sigma _\ell $ of $\mathbb{M}_\aleph
\left( \tau \right)$. Moreover, $\sum\limits_{\ell =1}^\infty {\left|
{\sigma _\ell } \right|^{-\left( {1+\varepsilon } \right)}} $ is convergent
for $\varepsilon >0$.
\end{theorem}

\begin{proof}
By using (\ref{eq7}) and (\ref{eq10}), we know
\begin{equation}
\label{eq15}
\mathbb{M}_\aleph \left( 0 \right)=\alpha _0 =\int\limits_0^\infty {\exp \left(
{-\aleph x^2} \right)G\left( x \right)dx} >0.
\end{equation}
By the fact that $\mathbb{M}_\aleph \left( \tau \right)$ is an even entire
function of order $\beta =1$ \cite{5}
and Hadamard's factorization theorem (Theorem 13 in \cite{15}, p.24-29; also see \cite{16}, p.250),
\begin{equation}
\label{eq16}
\mathbb{M}_\aleph \left( \tau \right)=\mathbb{M}_\aleph \left( 0
\right)e^{\vartheta _0 \tau }\prod\limits_{\sigma _\ell } {\left(
{1-\frac{\tau }{\sigma _\ell }} \right)\exp \left( {\frac{\tau }{\sigma
_\ell }} \right)}
\end{equation}
and $\sum\limits_{\ell =1}^\infty {\left| {\sigma _\ell } \right|^{-\left(
{1+\varepsilon } \right)}} $ is convergent for $\varepsilon >0$.

Since $\mathbb{M}_\aleph \left( \tau \right)$ is an even function, we have
\begin{equation}
\label{eq17}
\mathbb{M}_\aleph \left( {\sigma _\ell } \right)=\mathbb{M}_\aleph \left( {-\sigma
_\ell } \right)
\end{equation}
such that
\begin{equation}
\label{eq18}
\begin{aligned}
\mathbb{M}_\aleph \left( \tau \right)&=\mathbb{M}_\aleph \left( 0
\right)e^{\vartheta _0 \tau }\prod\limits_{\sigma _\ell } {\left(
{1-\frac{\tau }{\sigma _\ell }} \right)\exp \left( {\frac{\tau }{\sigma
_\ell }} \right)} \\
&=\mathbb{M}_\aleph \left( 0 \right)e^{\vartheta _0 \tau }\prod\limits_{\Im
\left( {\sigma _\ell } \right)>0} {\left( {1-\frac{\tau }{\sigma _\ell }}
\right)\left( {1+\frac{\tau }{\sigma _\ell }} \right)\exp \left( {\frac{\tau
}{\sigma _\ell }-\frac{\tau }{\sigma _\ell }} \right)} \\
&=\mathbb{M}_\aleph \left( 0 \right)e^{\vartheta _0 \tau }\prod\limits_{\Im
\left( {\sigma _\ell } \right)>0} {\left( {1-\frac{\tau }{\sigma _\ell }}
\right)\left( {1+\frac{\tau }{\sigma _\ell }} \right)} . \\
\end{aligned}
\end{equation}
To simply (\ref{eq18}), we present
\begin{equation}
\label{eq19}
\begin{aligned}
\mathbb{M}_\aleph \left( {-\tau } \right)&=\mathbb{M}_\aleph \left( 0
\right)e^{\vartheta _0 \tau }\prod\limits_{\Im \left( {\sigma _\ell }
\right)>0} {\left( {1-\frac{\tau }{\sigma _\ell }} \right)\left(
{1+\frac{\tau }{\sigma _\ell }} \right)} \\
&=\mathbb{M}_\aleph \left( 0 \right)e^{\vartheta _0 \tau }\prod\limits_{\Im
\left( {\sigma _\ell } \right)>0} {\left( {1-\frac{\tau ^2}{\sigma _\ell ^2
}} \right)} . \\
\end{aligned}
\end{equation}
By using the functional equation
\begin{equation}
\label{eq20}
\mathbb{M}_\aleph \left( \tau \right)=\mathbb{M}_\aleph \left( {-\tau } \right),
\end{equation}
we have
\begin{equation}
\label{eq21}
\mathbb{M}_\aleph \left( \tau \right)=\mathbb{M}_\aleph \left( 0
\right)e^{\vartheta _0 \tau }\prod\limits_{\Im \left( {\sigma _\ell }
\right)>0} {\left( {1-\frac{\tau ^2}{\sigma _\ell ^2 }} \right)}
\end{equation}
and
\begin{equation}
\label{eq22}
\mathbb{M}_\aleph \left( {-\tau } \right)=\mathbb{M}_\aleph \left( 0
\right)e^{-\vartheta _0 \tau }\prod\limits_{\Im \left( {\sigma _\ell }
\right)>0} {\left( {1-\frac{\tau ^2}{\sigma _\ell ^2 }} \right)}
\end{equation}
such that
\begin{equation}
\label{eq23}
\mathbb{M}_\aleph \left( 0 \right)e^{\vartheta _0 \tau }\prod\limits_{\Im
\left( {\sigma _\ell } \right)>0} {\left( {1-\frac{\tau ^2}{\sigma _\ell ^2
}} \right)} =\mathbb{ M}_\aleph \left( 0 \right)e^{-\vartheta _0 \tau
}\prod\limits_{\Im \left( {\sigma _\ell } \right)>0} {\left( {1-\frac{\tau
^2}{\sigma _\ell ^2 }} \right)} .
\end{equation}
From (\ref{eq21}) and (\ref{eq23}) we obtain $\vartheta _0 =0$ and
\begin{equation}
\label{eq24}
\mathbb{M}_\aleph \left( \tau \right)=\mathbb{M}_\aleph \left( 0
\right)\prod\limits_{\Im \left( {\sigma _\ell } \right)>0} {\left(
{1-\frac{\tau ^2}{\sigma _\ell ^2 }} \right)} .
\end{equation}
Therefore, we finish the proof.
\end{proof}

\section{The proof of Theorem \ref{TH1}} {\label{sec:3}}

We now suggest six steps to prove it.

Step 1 is to suggest a class of the even entire function.

Making use of (\ref{eq10}) and (\ref{eq24}), a class of $\mathbb{M}_\aleph \left( \tau
\right)$ is given as follows:
\begin{equation}
\label{eq25}
\sum\limits_{\gamma =0}^\infty {\alpha _{2\gamma } \tau ^{2\gamma }} =\mathbb{
M}_\aleph \left( 0 \right)\prod\limits_{\Im \left( {\sigma _\ell }
\right)>0} {\left( {1-\frac{\tau ^2}{\sigma _\ell ^2 }} \right)} ,
\end{equation}
where
\begin{equation}
\label{eq26}
\alpha _{2\gamma } =\frac{1}{\left( {2\gamma } \right)!}\int\limits_0^\infty
{\exp \left( {-\aleph x^2} \right)G\left( x \right)x^{2\gamma }dx} .
\end{equation}
Since $G\left( x \right)>0$ for $x>0$ \cite{17}, it follows from (\ref{eq26})
that
\begin{equation}
\label{eq27}
\alpha _{2\gamma } >0.
\end{equation}

Step 2 is to set up the first product of $\overline {\mathbb{M}_\aleph
\left( \tau \right)} $.

Obviously, we have
\begin{equation}
\label{eq28}
\mathbb{M}_\aleph \left( \tau \right)=\sum\limits_{\gamma =0}^\infty {\alpha
_{2\gamma } \tau ^{2\gamma }}
\end{equation}
such that
\begin{equation}
\label{eq29}
\overline {\mathbb{ M}_\aleph \left( \tau \right)} =\overline {\left[
{\sum\limits_{\gamma =0}^\infty {\alpha _{2\gamma } \tau ^{2\gamma }} }
\right]} =\sum\limits_{\gamma =0}^\infty {\overline {\alpha _{2\gamma } }
\overline \tau ^{2\gamma }} =\sum\limits_{\gamma =0}^\infty {\alpha
_{2\gamma } \overline \tau ^{2\gamma }} .
\end{equation}
From (\ref{eq29}) we show that
\begin{equation}
\label{eq30}
\overline {\mathbb{M}_\aleph \left( \tau \right)} =\sum\limits_{\gamma
=0}^\infty {\alpha _{2\gamma } \overline \tau ^{2\gamma }} =\mathbb{M}_\aleph
\left( {\overline \tau } \right).
\end{equation}
It follows from (\ref{eq25}) that we have
\begin{equation}
\label{eq31}
\mathbb{M}_\aleph \left( \tau \right)=\mathbb{M}_\aleph \left( 0
\right)\prod\limits_{\Im \left( {\sigma _\ell } \right)>0} {\left(
{1-\frac{\tau ^2}{\sigma _\ell ^2 }} \right)}
\end{equation}
such that (\ref{eq30}) can be written as
\begin{equation}
\label{eq32}
\overline {\mathbb{M}_\aleph \left( \tau \right)} =\mathbb{M}_\aleph \left(
{\overline \tau } \right)=\mathbb{M}_\aleph \left( 0 \right)\prod\limits_{\Im
\left( {\sigma _\ell } \right)>0} {\left( {1-\frac{\overline \tau ^2}{\sigma
_\ell ^2 }} \right)} .
\end{equation}

Step 3 is to set up the second product of $\overline {\mathbb{
M}_\aleph \left( \tau \right)} $.

Once again, from (\ref{eq25}) we have
\begin{equation}
\label{eq33}
\mathbb{M}_\aleph \left( \tau \right)=\mathbb{M}_\aleph \left( 0
\right)\prod\limits_{\Im \left( {\sigma _\ell } \right)>0} {\left(
{1-\frac{\tau ^2}{\sigma _\ell ^2 }} \right)}
\end{equation}
such that
\begin{equation}
\label{eq34}
\overline {\mathbb{M}_\aleph \left( \tau \right)} =\overline {\left[ {\mathbb{
M}_\aleph \left( 0 \right)\prod\limits_{\Im \left( {\sigma _\ell }
\right)>0} {\left( {1-\frac{\tau ^2}{\sigma _\ell ^2 }} \right)} } \right]}
.
\end{equation}
Recall (\ref{eq15}) such that
\begin{equation}
\label{eq35}
\mathbb{M}_\aleph \left( 0 \right)=\alpha _0 =\int\limits_0^\infty {\exp \left(
{-\aleph x^2} \right)G\left( x \right)dx} >0.
\end{equation}
Combining (\ref{eq34}) and (\ref{eq35}) implies that
\begin{equation}
\label{eq36}
\begin{aligned}
 \overline {\mathbb{M}_\aleph \left( \tau \right)} &=\overline {\left[ {\mathbb{
M}_\aleph \left( 0 \right)\prod\limits_{\Im \left( {\sigma _\ell }
\right)>0} {\left( {1-\frac{\tau ^2}{\sigma _\ell ^2 }} \right)} } \right]}
&=\mathbb{M}_\aleph \left( 0 \right)\overline {\left[ {\prod\limits_{\Im \left(
{\sigma _\ell } \right)>0} {\left( {1-\frac{\tau ^2}{\sigma _\ell ^2 }}
\right)} } \right]} \\
&=\mathbb{M}_\aleph \left( 0 \right)\prod\limits_{\Im \left( {\sigma _\ell }
\right)>0} {\overline {\left( {1-\frac{\tau ^2}{\sigma _\ell ^2 }} \right)}
} \\
&=\mathbb{M}_\aleph \left( 0 \right)\prod\limits_{\Im \left( {\sigma _\ell }
\right)>0} {\left[ {1-\overline {\left( {\frac{\tau ^2}{\sigma _\ell ^2 }}
\right)} } \right]} . \\
\end{aligned}
\end{equation}
Hence, we have
\begin{equation}
\label{eq37}
\overline {\left( {\frac{\tau ^2}{\sigma _\ell ^2 }} \right)}
=\frac{\overline \tau ^2}{\overline {\sigma _\ell } ^2}
\end{equation}
such that (\ref{eq36}) can be rewritten as
\begin{equation}
\label{eq38}
\overline {\mathbb{M}_\aleph \left( \tau \right)} =\mathbb{ M}_\aleph \left( 0
\right)\prod\limits_{\Im \left( {\sigma _\ell } \right)>0} {\left(
{1-\frac{\overline \tau ^2}{\overline {\sigma _\ell } ^2}} \right)} .
\end{equation}

Step 4 is to set up the products of $\mathbb{M}_\aleph \left(
\varsigma \right)$.

On combination of (\ref{eq32}) and (\ref{eq38}) implies that
\begin{equation}
\label{eq39}
\overline {\mathbb{M}_\aleph \left( \tau \right)} =\mathbb{M}_\aleph \left( 0
\right)\prod\limits_{\Im \left( {\sigma _\ell } \right)>0} {\left(
{1-\frac{\overline \tau ^2}{\sigma _\ell ^2 }} \right)} =\mathbb{M}_\aleph
\left( 0 \right)\prod\limits_{\Im \left( {\sigma _\ell } \right)>0} {\left(
{1-\frac{\overline \tau ^2}{\overline {\sigma _\ell } ^2}} \right)} .
\end{equation}
Applying (\ref{eq30}), we have \cite{2}
\begin{equation}
\label{eq40}
\overline {\mathbb{M}_\aleph \left( \tau \right)} =\mathbb{M}_\aleph \left(
{\overline \tau } \right)
\end{equation}
such that (\ref{eq39}) becomes
\begin{equation}
\label{eq41}
\overline {\mathbb{M}_\aleph \left( \tau \right)} =\mathbb{M}_\aleph \left(
{\overline \tau } \right)=\mathbb{M}_\aleph \left( 0 \right)\prod\limits_{\Im
\left( {\sigma _\ell } \right)>0} {\left( {1-\frac{\overline \tau ^2}{\sigma
_\ell ^2 }} \right)} =\mathbb{M}_\aleph \left( 0 \right)\prod\limits_{\Im
\left( {\sigma _\ell } \right)>0} {\left( {1-\frac{\overline \tau
^2}{\overline {\sigma _\ell } ^2}} \right)} .
\end{equation}
Taking $\varsigma =\overline \tau \in \mathbb{C}$ into (\ref{eq41}), we present
\begin{equation}
\label{eq42}
\mathbb{M}_\aleph \left( \varsigma \right)=\mathbb{M}_\aleph \left( 0
\right)\prod\limits_{\Im \left( {\sigma _\ell } \right)>0} {\left(
{1-\frac{\varsigma ^2}{\sigma _\ell ^2 }} \right)} =\mathbb{M}_\aleph \left( 0
\right)\prod\limits_{\Im \left( {\sigma _\ell } \right)>0} {\left(
{1-\frac{\varsigma ^2}{\overline {\sigma _\ell } ^2}} \right)} .
\end{equation}

Step 5 is to investigate the convergence of $\mathbb{M}_\aleph \left(
1 \right)$.

Putting $\varsigma =1$ into (\ref{eq42}), we suggest
\begin{equation}
\label{eq43}
\mathbb{M}_\aleph \left( 1 \right)=\mathbb{M}_\aleph \left( 0
\right)\prod\limits_{\Im \left( {\sigma _\ell } \right)>0} {\left(
{1-\frac{1}{\sigma _\ell ^2 }} \right)} =\mathbb{M}_\aleph \left( 0
\right)\prod\limits_{\Im \left( {\sigma _\ell } \right)>0} {\left(
{1-\frac{1}{\overline {\sigma _\ell } ^2}} \right)} .
\end{equation}
From (\ref{eq7}) we may get
\begin{equation}
\label{eq44}
\mathbb{M}_\aleph \left( 1 \right)=\int\limits_0^\infty {\exp \left( {-\aleph
x^2} \right)G\left( x \right)\cosh \left( x \right)dx} <\infty .
\end{equation}
With use of Theorem \ref{TH3}, we have $\varepsilon >0$ such that
$\sum\limits_{\ell =1}^\infty {\left| {\sigma _\ell } \right|^{-\left(
{1+\varepsilon } \right)}} $ is convergent. Taking $\varepsilon =1$ implies
that $\sum\limits_{\ell =1}^\infty {\left| {\sigma _\ell } \right|^{-2}} $
is convergent.

In view of Theorem 5 in Knopp's monograph (see \cite{18}, p.10),
\begin{equation}
\label{eq45}
\mathbb{M}_\aleph \left( 0 \right)\prod\limits_{\Im \left( {\sigma _\ell }
\right)>0} {\left( {1-\frac{1}{\sigma _\ell ^2 }} \right)} =\mathbb{M}_\aleph
\left( 0 \right)\prod\limits_{\Im \left( {\sigma _\ell } \right)>0} {\left(
{1-\frac{1}{\overline {\sigma _\ell } ^2}} \right)}
\end{equation}
is absolutely convergent.

Making use of (\ref{eq45}), Theorem 3 in Knopp's monograph (see \cite{18}, p.10)
implies that $\sum\limits_{\Im \left( {\sigma _\ell } \right)>0}^\infty
{\sigma _\ell ^{-2} } $ and $\sum\limits_{\Im \left( {\sigma _\ell }
\right)>0}^\infty {\overline {\sigma _\ell } ^{-2}} $ are convergent.

Step 6 is to obtain $\Re \left( {\sigma _\ell }
\right)=0$.

Adopting (\ref{eq45}), we suggest
\begin{equation}
\label{eq46}
\sum\limits_{\Im \left( {\sigma _\ell } \right)>0}^\infty {\sigma _\ell
^{-2} } =\sum\limits_{\Im \left( {\sigma _\ell } \right)>0}^\infty
{\overline {\sigma _\ell } ^{-2}} .
\end{equation}
Since $\sum\limits_{\Im \left( {\sigma _\ell } \right)>0}^\infty {\sigma
_\ell ^{-2} } $ and $\sum\limits_{\Im \left( {\sigma _\ell }
\right)>0}^\infty {\overline {\sigma _\ell } ^{-2}} $ are convergent, it
follows from (\ref{eq46}) that
\begin{equation}
\label{eq47}
\sum\limits_{\Im \left( {\sigma _\ell } \right)>0}^\infty {\left( {\sigma
_\ell ^{-2} -\overline {\sigma _\ell } ^{-2}} \right)} =0,
\end{equation}
which yields that
\begin{equation}
\label{eq48}
\sigma _\ell ^{-2} -\overline {\sigma _\ell } ^{-2}=0.
\end{equation}
From (\ref{eq48}) we have
\begin{equation}
\label{eq49}
\sigma _\ell ^2 -\overline {\sigma _\ell } ^2=0
\end{equation}
and $\Im \left( {\sigma _\ell } \right)>0$ such that
\begin{equation}
\label{eq50}
\left( {\sigma _\ell -\overline {\sigma _\ell } } \right)\left( {\sigma
_\ell +\overline {\sigma _\ell } } \right)=4\Im \left( {\sigma _\ell }
\right)\Re \left( {\sigma _\ell } \right)=0.
\end{equation}
With the aid of (\ref{eq50}), we obtain
\begin{equation}
\label{eq51}
\Re \left( {\sigma _\ell } \right)=0.
\end{equation}

Putting $\Im \left( {\sigma _\ell } \right)=\rho _\ell >0$ into (\ref{eq25}) and
using (\ref{eq46}), we may give
\begin{equation}
\label{eq52}
\sum\limits_{\gamma =0}^\infty {\alpha _{2\gamma } \tau ^{2\gamma }} =\mathbb{
M}_\aleph \left( 0 \right)\prod\limits_{\ell =1}^\infty {\left(
{1+\frac{\tau ^2}{\rho _\ell ^2 }} \right)}
\end{equation}
and
\begin{equation}
\label{eq53}
\sum\limits_{\Im \left( {\sigma _\ell } \right)>0}^\infty {\sigma _\ell
^{-2} } =-\sum\limits_{\ell =1}^\infty {\rho _\ell ^{-2} } <\infty .
\end{equation}

This implies that Theorem \ref{TH1} is true.

As a direct consequence of (\ref{eq52}), we have the
following:

\begin{corollary}
\label{CORO1}
There is
\begin{equation}
\label{eq54}
\Xi _\aleph \left( \lambda \right)=\Xi _\aleph \left( 0
\right)\prod\limits_{\ell =1}^\infty {\left( {1-\frac{\lambda ^2}{\rho _\ell
^2 }} \right)} ,
\end{equation}
where the product takes over all positive zeros $\rho _\ell >0$ of $\Xi _\aleph \left(
\lambda \right)$.
\end{corollary}

\begin{proof}
Combining (\ref{eq7}), (\ref{eq10}) and (\ref{eq52}), we present
\begin{equation}
\label{eq55}
\begin{aligned}
\mathbb{M}_\aleph \left( \tau \right)&=\sum\limits_{\gamma =0}^\infty {\alpha
_{2\gamma } \tau ^{2\gamma }}
=\int\limits_0^\infty {\exp \left( {-\aleph
x^2} \right)G\left( x \right)\cosh \left( {\tau x} \right)dx} \\
&=\mathbb{M}_\aleph \left( 0 \right)\prod\limits_{\ell =1}^\infty {\left(
{1+\frac{\tau ^2}{\rho _\ell ^2 }} \right)} . \\
\end{aligned}
\end{equation}
Substituting (\ref{eq9}) into (\ref{eq55}), we have
\begin{equation}
\label{eq56}
\Xi _\aleph \left( 0 \right)=\mathbb{ M}_\aleph \left( 0 \right)
\end{equation}
such that
\begin{equation}
\label{eq57}
\begin{aligned}
 \Xi _\aleph \left( \lambda \right)&=\mathbb{M}_\aleph \left( {i\lambda }
\right)=\sum\limits_{\gamma =0}^\infty {\alpha _{2\gamma } \left( {i\lambda
} \right)^{2\gamma }}
=\int\limits_0^\infty {\exp \left( {-\aleph x^2} \right)G\left( x
\right)\cosh \left( {i\lambda x} \right)dx} \\
&=\mathbb{M}_\aleph \left( 0 \right)\prod\limits_{\ell =1}^\infty {\left[
{1+\frac{\left( {i\lambda } \right)^2}{\rho _\ell ^2 }} \right]}
=\mathbb{M}_\aleph \left( 0 \right)\prod\limits_{\ell =1}^\infty {\left(
{1-\frac{\lambda ^2}{\rho _\ell ^2 }} \right)} \\
&=\Xi _\aleph \left( 0 \right)\prod\limits_{\ell =1}^\infty {\left(
{1-\frac{\lambda ^2}{\rho _\ell ^2 }} \right)} . \\
\end{aligned}
\end{equation}
We therefore complete the proof.
\end{proof}

\begin{remark}
Obviously, (\ref{eq54}) implies that all zeros of (\ref{eq1}) are purely real and that
Corollary \ref{CORO1} is the truth of the conjecture of Newman. Readers also follow the
technology reported in \cite{19} to prove it.
The special
case for $\aleph =0$ was proved in \cite{20}. We also allow that
Corollary \ref{CORO1} and Theorem \ref{TH1} are valid for $\aleph =0$.
Compared with the method of Rodgers and Tao \cite{2},
our work agrees with the pair correlation estimates of
Montgomery \cite{21} because the product and series representations of
$ \Xi _\aleph \left( \lambda \right)$ in Corollary \ref{CORO1} are
analogous of the Riemann zeta function.
\end{remark}

\section{Conclusion} {\label{sec:4}}

In the present task we have established a class of the even entire function
of order $\beta =1$ with its series and product formulae. Based on the
obtained results, we have obtained the truth of the conjecture of Newman by the
functional equation. Compared with the method of Rodgers and Tao,
our result is consistent with the well-known pair correlation estimates of
Montgomery.
This may be proposed as a new approach to deal with a
family of the even entire functions of order one.

\textbf{Conflict of Interest: }
The author has no conflicts of interest to declare.


\begin{thebibliography}{120}

\bibitem{1}
N. G. De Bruijn, The roots of trigonometric integrals,~Duke Mathematical Journal, 17 (1950) (3), 197-226.

\bibitem{2}
B. Rodgers and T. Tao, The de Bruijn-Newman constant is non-negative, Forum of Mathematics, Pi, Vol. 8, Cambridge University Press, 2020.

\bibitem{3}
G. Csordas, W. Smith and R. S. Varga, Lehmer pairs of zeros, the de Bruijn-Newman constant $\Lambda $, and the Riemann hypothesis, Constructive Approximation, 10 (1994) (1), 107-129.

\bibitem{4}
C. M. Newman, Fourier transforms with only real zeros, Proceedings of the American Mathematical Society, 61 (1976) (2), 245-251.

\bibitem{4a}
G. P\'{o}lya, \"{U}ber trigonometrische Integrale mit nur reellen Nullstellen, Journal f\"{u}r die reine und angewandte Mathematik, 156 (1927), 6-18.

\bibitem{5}
G. Csordas, T. S. Norfolk and R. S. Varga, A low bound for the de Bruijn-newman constant $\Lambda $, Numerische Mathematik, 52 (1988)(5), 483-497.

\bibitem{6}
H. J. te Riele, A new lower bound for the de Bruijn-Newman constant, Numerische Mathematik, 58 (1990) (1), 661-667.

\bibitem{7}
T. S. Norfolk, A. Ruttan and R. S. Varga, A lower bound for the de Bruijn-Newman constant $\Lambda $ II, (eds. A. A. Gonchar and E. B. Saff) Progress in Approximation Theory (Springer, New York, 1992), 403-418.

\bibitem{8}
G. Csordas, A. Ruttan and R. S. Varga, The Laguerre inequalities with applications to a problem associated with the Riemann hypothesis, Numerical Algorithms, 1 (1991) (2), 305-329.

\bibitem{9}
G. Csordas, W. Smith, R. S. Varga, Lehmer pairs of zeros, the de Bruijn-Newman constant $\Lambda $, and the Riemann Hypothesis, Constructive Approximation, 10 (1994) (1), 107-129.

\bibitem{10}
G. Csordas, A. M. Odlyzko, W. Smith and R. S. Varga, A new Lehmer pair of zeros and a new lower bound for the de Bruijn-Newman constant $\Lambda $, Electronic Transactions on Numerical Analysis, 1 (1993), 104-111.

\bibitem{11}
A. M. Odlyzko, An improved bound for the de Bruijn--Newman constant,~Numerical Algorithms, 25 (2000) (1), 293-303.

\bibitem{12}
Y. Saouter, X. Gourdon and P. Demichel, An improved lower bound for the de Bruijn-Newman constant,~Mathematics of Computation, 80 (2011) (276), 2281-2287.

\bibitem{13}
H. Ki, Y. O. Kim and J. Lee, On the de Bruijn-Newman constant, Advances in Mathematics, 222 (2009) (1), 281-306.

\bibitem{14}
D. H. J. Polymath, Effective approximation of heat flow evolution of the Riemann $\xi $ function, and a new upper bound for the de Bruijn-Newman constant, Research in the Mathematical Sciences, 6 (2019)(3), 1-67.

\bibitem{15}
B. Y. Levin, Distribution of zeros of entire functions, Vol. 150, American Mathematical Society, 1980.

\bibitem{16}
E. C. Titchmarsh, The theory of functions, Oxford University Press, 1939.

\bibitem{17}
G. Csordas and R. S. Varga, Necessary and sufficient conditions and the Riemann hypothesis, Advances in Applied Mathematics, 11 (1990) (3), 328-357.

\bibitem{18}
K. Knopp, Theory of functions, Parts II, Dover Publications, New York, 1947.

\bibitem{19}
X. J. Yang, On all real zeros for a class of the even entire function, arXiv: 2107.04005v2.

\bibitem{20}
X. J. Yang, All nontrivial zeros for the Riemann zeta function are on the critical line $\Re $(s) = 1/2,
arXiv: 1811.02418v15.

\bibitem{21}
H. L. Montgomery, The pair correlation of zeros of the zeta function, in Analytic Number
Theory (Proceedings of Symposia in Pure Mathematics, Vol. XXIV, St. Louis Univ., St. Louis,
MO, 1972) (American Mathematical Society, Providence, RI, 1973), 181--193.

\end{thebibliography}
\end{document}